\documentclass[reqno, 10pt]{amsart}
\usepackage{bbm}
\usepackage{mathbbol}
\usepackage{amsfonts}
\usepackage{mathrsfs}
\usepackage{amscd}
\usepackage{cases}
\usepackage{mathrsfs}

\usepackage{subfig} % by Simone
\usepackage{graphicx}
\allowdisplaybreaks
\usepackage[top=4cm, bottom=3cm, left=3.2cm, right=3.2cm]{geometry}
\newtheorem{thm}{Theorem}

\newtheorem{lem}{Lemma}

\theoremstyle{definition}

\theoremstyle{remark}
\newtheorem{rem}{Remark}
\numberwithin{equation}{section} \numberwithin{lem}{section}
\numberwithin{thm}{section} \numberwithin{prop}{section}
\numberwithin{cor}{section} \numberwithin{rem}{section}\numberwithin{hyp}{section}

\begin{document}

\title[
Euler alignment system]{
The global classical solution to compressible Euler system with velocity alignment}

\author{Lining Tong$^{\uppercase{1}}$}\thanks{$^1$ Email:  tongln@shu.edu.cn; Department of Mathematics, Shanghai University, 200244, China; Lining Tong is supported by NSFC(No. 11771274).}

\author{Li Chen$^{\uppercase{2}}$}\thanks{$^2$ Email: chen@math.uni-mannheim.de; Department of Mathematics, University of Mannheim, 68131 Mannheim, Germany; }

\author{Simone G\"ottlich$^{\uppercase{3}}$}\thanks{$^3$ Email: goettlich@uni-mannheim.de;  Department of Mathematics, University of Mannheim, 68131 Mannheim, Germany; Simone G\"ottlich is supported by the German Research Foundation, DFG grant GO 1920/7-1.}

\author{Shu Wang$^{\uppercase{4}}$}\thanks{$^4$ Email:wangshu@bjut.edu.cn; College of Applied Sciences, Beijing University of Technology, Ping Le Yuan 100, Beijing 100124, China; Shu Wang is supported by NSFC(No. 11831003, 11771031, 11531010).}
\thanks{$_*$ Corresponding author.}

\date{\today}

\maketitle

\begin{abstract}
 In this paper, the compressible Euler system with velocity alignment and damping is  considered, where the influence matrix of velocity alignment is not positive definite. Sound speed is used to reformulate the system into symmetric hyperbolic type. The global existence and uniqueness
of smooth solution for small initial data is provided.
\end{abstract}

{\tiny AMS SUBJECT CLASSIFICATIONS:} 35Q70, 35L65

{\tiny KEYWORDS:} non-local velocity alignment, global existence, nonlinear pressure%long-time behavior.

%%%%%%%%%%%%%%%%%%%%%%%%%%%%%%%%%%%%%%%%%%%%%%%%%%
\section{Introduction}
%%%%%%%%%%%%%%%%%%%%%%%%%%%%%%%%%%%%%%%%%%%%%%%%%%

In this paper, we study the following Cauchy problem
\begin{eqnarray}
&&\label{eqn1}\partial_t \rho+\nabla\cdot(\rho u)=0,\\
&&\label{eqn2}\partial_t (\rho u)+\nabla\cdot(\rho u\otimes u)+\nabla p(\rho)=-\frac{1}{\tau}\rho u-a\rho\int_{\mathbb{R}^N}{\mathbf \Gamma}(x-y)(u(x)-u(y))\rho(y)dy,
\end{eqnarray}
in $(0,\infty)\times \mathbb{R}^N$  and the initial conditions
\begin{eqnarray}\label{eqn3}
\rho|_{t=0}=\rho_0(x),\quad u|_{t=0}=u_0(x),\quad x\in \mathbb{R}^N,
\end{eqnarray}
where $\rho$ and $u$ are the unknown density and velocity, and the pressure $P(\rho)=A\rho^{\gamma}.$ The matrix is ${\mathbf \Gamma}(x)\in L^1(\mathbb{R}^N)$. The constants $A, \gamma\geq1, \tau>0, a>0$ are given. For simplicity, it is assumed that $A=1$.

For quasi-linear  hyperbolic systems, intensive studies have been carried out in numerous literature, such as \cite{K,L,M,MN}, and etc. It has been proved in general that for symmetrizable hyperbolic systems, smooth solutions exist locally in time and shock waves formation will breakdown the smoothness of the solutions in finite time even for the scalar case, see \cite{M}. However, if some damping terms are taken into account, shock waves can be avoided for small perturbation of the diffusion waves \cite{GHL, HTL}. Moreover, with the help of velocity damping, the existence and uniqueness of the global classical solution can be obtained, see for example \cite{HMP, STW, TW, WY, YZZ}.

The pressureless Euler system with non-local forces has been studied recently and very limited results have been done. The local existence of classical solutions of the complex material flow dynamics which has been derived in \cite{GKT}, under structural condition for the interaction force has been obtained in  \cite{CCGW}.   In \cite{CCZ}, for  the 1-D model with damping and non-local interaction, a critical threshold for the existence of classical solution by using the characteristic method is presented. 1-D entropy weak solution for Cucker-Smale type interaction has been obtained with the help of the compensated compactness argument in \cite{HHW}. The global existence of smooth solutions with small initial data for the model with velocity alignment can be found in \cite{CCTT, DKRT, HKK, HKK2, KT}. In these references, the influence function of velocity alignment is $\mathbf \Gamma (x)=\phi(|x|)\mathbf I_{N\times N}$, where $\phi(|x|)$ has a positive lower bound and $\mathbf I_{N\times N}$ is the identity matrix, with which the formation of shock can be prevented. In case that additional pressure and viscosity are added, for restrictive interaction potentials, the global weak solution and its long time behavior are obtained in \cite{CWZ}. By adapting the convex integration method, it has been shown in \cite{CFGG} that infinitely many weak solutions exist. Recently, the global existence of classical solutions for the hydrodynamic model with linear pressure term and non-local velocity alignment was given in \cite{C}, where the shock wave was  prevented by velocity alignment.

In many models, the communication weight matrices have different structures. Many of them do not need to be positive definite, as for example in the material flow model that has been proposed in \cite{GKT}, the interaction force includes $\mathbf\Gamma=\frac{x}{|x|}\otimes\frac{x}{|x|}$ as the weight of velocity alignment.

In our model, the influence function of velocity alignment $\mathbf \Gamma$ is a matrix which corresponds to a linear projection of the velocity field. Furthermore it is non-constant and not positive definite, which reflects the anisotropic non-local interaction within the system. Therefore, the velocity alignment alone can not prevent the formation of shock wave. In order to obtain the global existence of smooth solutions, the additional damping effect in the system is necessary. Additionally, the symmetry of the coefficients plays an important role in the analysis of the existence of smooth solutions when using the method of standard energy estimates. We will use sound speed to reconstruct equation $(\ref{eqn1})-(\ref{eqn2})$ into symmetric hyperbolic equations. It should be pointed out that this method will make the term of velocity alignment more complicated. After a detailed analysis of the relationship between velocity alignment and damping, the anisotropic non-local interaction  is overcome by using damping, and the existence %well-posedness
 of the global classical solution of problem $(\ref{eqn1})-(\ref{eqn2})$ is obtained.

Here, we introduce several notations used throughout the paper.  For a function $u=u(x)$, $\|u\|_{L^p}$ denotes the usual $L^p(\mathbb{R}^N)$ -norm. We also set $C$ as a generic positive constant independent of $t$. For any non-negative integer $k,\  H^k := H^k( \mathbb{R}^N)$ denotes the Sobolev space $=W^{k,2}(\mathbb{R}^N)$, and $C^k(I;E)$ is the space of $k$-times continuously differentiable
functions from an interval $I \subset \mathbb{R }$ into a Banach space $E$. $\nabla^k$ denotes any partial
derivative $\partial^\alpha$ with multi-index $\alpha$, where$|\alpha|=k$. For simplicity, we write $\int f\mathrm{d}x:=\int_{\mathbb{R}^N} f\mathrm{d}x$ .

This paper is structured as follows: In Section 2, we reformulate the Cauchy problem $(\ref{eqn1})-(\ref{eqn3})$ into a symmetric hyperbolic system and present our main result. In Section 3,  we demonstrate the local existence under uniqueness of the classical solution for  the reconstructed system. Finally, we establish the priori estimates to prove the global existence result.

\section{reformulation and main result }

\subsection{Reformulation of the problem}\label{2.1}In this subsection, we will reformulate the Cauchy problem of the compressible Euler system $(\ref{eqn1})-(\ref{eqn3})$ as in \cite{STW}.
the main point is to obtain a symmetric system. We consider the case $\gamma>1$ in this paper and introduce the sound speed:
$$
\kappa(\rho) =\sqrt{P^{\prime}(\rho)},
$$
where $\bar{\kappa}=\kappa(\bar{\rho})$ is set to the sound speed at a background density $\bar{\rho}>0$.
The symmetrization in the case of $\gamma=1$ can be done similarly with a new variable $\ln \rho$.

Define
$$
\sigma(\rho)%=\frac{2}{\gamma-1}(\kappa(\rho)-\bar{\kappa})
=\nu(\kappa(\rho)-\bar{\kappa})\ \mathrm{ with }\ \nu=\frac{2}{\gamma-1}.
$$
Then the equation $(\ref{eqn1})-(\ref{eqn2})$ are transformed into the following system:
\begin{align}
&\label{eqn1a}\partial_t \sigma+\bar{\kappa}\nabla\cdot u =-u\cdot \nabla \sigma-\frac{1}{\nu} \sigma\nabla\cdot u,\\
&\label{eqn2a}\partial_t u+\bar{\kappa}\nabla \sigma+\frac{1}{\tau} u=-u\cdot \nabla u-\frac{1}{\nu}\sigma\nabla\cdot \sigma-a\int\mathbf \Gamma(x-y)(u(x)-u(y))(\frac{1}{\nu} \sigma(y)+\bar{\kappa})^\nu \mathrm{d}y,
\end{align}
where the constants are $a=\gamma^{-1/{(\gamma-1)}}>0$. The initial condition (\ref{eqn3}) becomes
\begin{equation}
\label{eqn3a}(\sigma, u)|_{t=0}=(\sigma_0(x),u_0(x))
\end{equation}
with $\sigma_0=\nu(\kappa(\rho_0)-\bar{\kappa})$. %with
%$$
%n_0=2(\kappa(\rho_0)-\bar{\kappa})
%$$
Note that as we did at the formal level, we can find the relation between the classical solutions $(\rho, u)$ and $(\sigma, u)$ to the systems $(\ref{eqn1})-(\ref{eqn2})$ and (\ref{eqn1a})-(\ref{eqn2a}), respectively, in
the following two lemmas. The proofs can be obtained by taking the similar strategy as in \cite{STW}.
\begin{lem}
For any $T>0$, if $(\rho, u)\in C^1(\mathbb{R}^N\times [0,T])$ is a solution of system $(\ref{eqn1})-(\ref{eqn2})$ with $\rho>0$, then $(\sigma,u))\in C^1(\mathbb{R}^N\times [0,T])$ is a solution for the system $(\ref{eqn1a})-(\ref{eqn2a})$ with $(\frac{1}{\nu} \sigma+\bar{\kappa})^\nu >0$. Conversely, if  $(\sigma,u))\in C^1(\mathbb{R}^N\times [0,T])$ is a solution for the system $(\ref{eqn1a})-(\ref{eqn2a})$ with $(\frac{1}{\nu} \sigma+\bar{\kappa})^\nu >0$, then $(\rho, u)\in C^1(\mathbb{R}^N\times [0,T])$ is a solution of system $(\ref{eqn1})-(\ref{eqn2})$ with $\rho>0$.
\end{lem}
\begin{lem}
For any $T>0$, if $(\rho, u)\in C^1(\mathbb{R}^N\times [0,T])$ is a uniformly bounded solution of system $(\ref{eqn1})-(\ref{eqn2})$ with $\rho_0>0$, then $\rho>0$ on $\mathbb{R}^N\times [0,T]$. Conversely, if  $(\sigma,u))\in C^1(\mathbb{R}^N\times [0,T])$ is a uniformly bounded solution of system $(\ref{eqn1a})-(\ref{eqn2a})$ with $(\frac{1}{\nu} \sigma_0+\bar{\kappa})^\nu >0$, then  $(\frac{1}{\nu} \sigma+\bar{\kappa})^\nu >0$ on $\mathbb{R}^N\times [0,T]$.
\end{lem}

\subsection{Main result}

In the subsection \ref{2.1} , we have presented the equivalent reconstruction system (\ref{eqn1a})-(\ref{eqn2a}) of the problem (\ref{eqn1})-(\ref{eqn2}) and the equivalence relation between them. Next, we will study the reconstructed system (\ref{eqn1a})-(\ref{eqn2a}) and provide the following results.
\begin{thm} \label{local}{\rm(Local-in-time exitence)} For $s>\frac{N}{2}+1$, assume the initial values
 $(\sigma_0(x),u_0(x))\in H^{s}( \mathbb{R}^N).
$
Then there exist a unique classical solution $(\sigma,u)$ of the Cauchy problem $(\ref{eqn1a})-(\ref{eqn3a})$ satisfying
\begin{equation}
(\sigma,u)\in C([0,T],H^s(\mathbb{R}^N))\cap C^1([0,T],H^2(\mathbb{R}^N))
\end{equation}
for some finite $T>0$.%has a  satisfying
%\begin{equation}
%sup_{0\leq t\leq T_0}\|n\|_{H^3}+\|u\|_{H^3}<M_2.
%\end{equation}.
%$$(\rho,u)(t,x)\in C^1([0,T]\times  R^N)
%$$
\end{thm}
\begin{thm}\label{global}(Global-in-time exitence)
Suppose background sound speed $\bar{\kappa}$ satisfying $2a\bar{\kappa}^\nu\|\mathbf \Gamma\|_{L^1}<\frac{1}{\tau}$. If $\|\sigma_0\|_{H^s}+\|u_0\|_{H^s}\leq \delta_0$ with sufficiently small $\delta_0>0$, then the Cauchy problem $(\ref{eqn1a})-(\ref{eqn3a})$ has a unique global classical solution. %furthermore, for $p>2$,$\exists\  C>0$, such that
%\begin{equation}
%\|\sigma\|_{L^p}+\|\nabla\sigma(\cdot, t)\|_{H^{s-1}}+\|u(\cdot,t)\|_{H^s}\leq Ce^{-Ct}.
%\end{equation}
\end{thm}
\begin{rem}\label{rem}
Since the matrix $\mathbf \Gamma(x)$ is not positive definite, the damping coefficient needs to be large enough to make the damping term restrain the self-acceleration effect caused by velocity alignment to get the global well-posedness. The definition of condition $2a\bar{\kappa}^\nu\|\mathbf \Gamma\|_{L^1}<\frac{1}{\tau}$ is therefore natural. %And under the combined action of velocity alignment and damping, the group velocity of the system $u\rightarrow 0$ in long time.
\end{rem}

%%%%%%%%%%%%%%%%%%%%%%%%%%%%%%%%%%%%%%%%%%%%%%%%%%
\section{Local existence and uniqueness  }
%%%%%%%%%%%%%%%%%%%%%%%%%%%%%%%%%%%%%%%%%%%%%%%%%%
In this section, we demonstrate the local existence and uniqueness of the classical solutions to (\ref{eqn1a})-(\ref{eqn3a}).
We will present a successive iteration scheme to construct approximate solutions and to obtain the energy estimates. Then we show that approximate solutions are convergent in Sobolev spaces using the contraction mapping principle and prove that the limit function is the local solution.

\subsection{Approximate solutions.}
%To prove the existence for this particular case,
We construct approximate solution by the following iterative method:
\begin{itemize}
  \item the zeroth approximation: $(\sigma^0,\ u^0)(x,\ t)=(\sigma_0,\ u_0)$;
  \item Suppose that the $k$th approximation $(\sigma^k,\ u^k)(x,\ t),\ k\geq1$ is given. Then define the
   $(k+1)$th approximation $(\sigma^{k+1},\ u^{k+1})(x,\ t)$ as a solution of the linear system
\end{itemize}
\begin{align}
\label{eqnim1}\partial_t \sigma^{k+1}&+\bar{\kappa}\nabla\cdot u^{k+1} =-u^{k}\cdot \nabla \sigma^{k+1}-\frac{1}{\nu}\sigma^{k}\nabla\cdot u^{k+1},\\
\label{eqnim2}\partial_t u^{k+1}&+\bar{\kappa}\nabla \sigma^{k+1}+\frac{1}{\tau} u^{k+1}=-u^{k}\cdot \nabla u^{k+1}-\frac{1}{\nu}\sigma^{k}\nabla\cdot \sigma^{k+1}\nonumber\\
&-a\int \mathbf \Gamma(x-y)\Big(u^k(x)-u^k(y)\Big)
\Big(\frac{1}{\nu}\sigma^k(y)+\bar{\kappa}\Big)^\nu dy
\end{align}
with the initial data
\begin{equation}\label{eqnim3}
(\sigma^{k+1},\ u^{k+1})|_{t=0}=(\sigma_0(x),\ u_0(x))\in\ H^s(\mathbb{R}^N).
\end{equation}
The local existence of the solutions $(\sigma^{k+1}, u^{k+1})$ in Sobolev spaces can be obtained by applying the linear theory of the multi-dimensional hyperbolic equations in \cite{BGSD}.
\subsection{Priori estimates }

We first set up several constants:
\begin{equation}\label{ML}
%M=\sqrt{\|u\|_{H^{s+1}}^2+\delta_1},\quad L=\sqrt{\|\rho\|_{H^s}^2+\delta_2}
M=\sqrt{\|\sigma_0\|_{H^{s}}^2+\|u_0\|_{H^s}^2+1},
\end{equation}
and choose $T_0>0$ so that
%\begin{equation}\label{T0}
%e^{CT_0M}\|\rho_{0}\|_{H^s}^2+e^{CMT_0}(\|u_0\|_{H^{s+1}}^2+C(M^2+1)T_0(1+\|\rho_0\|_{L^1}^2))+e^{C(M+L+ML)T_0}(\| u_0\|_{H^{s+1}}^2+C(M+L+ML)T_0)<min\{\delta_1\, \delta_2\}
%\end{equation}
\begin{align}\label{T0}
%&e^{CT_0M}(\|\rho_{0}\|_{H^s}^2+C(M^2+1)T_0(1+\|\rho_0\|_{L^1}^2))\leq min\{\delta_1\, \delta_2\}\nonumber\\
%&e^{CMT_0}\|u_0\|_{H^{s+1}}^2+e^{C(M+L+ML)T_0}(\| u_0\|_{H^{s+1}}^2+C(M+L+ML)T_0)<min\{\delta_1\, \delta_2\}
(e^{C(M,\nu,\bar{\kappa})T_0}-1)\Big(\|\sigma_0\|_{H^s}^2+\|u_0\|_{H^s}^2\Big)+e^{C(M,\nu,\bar{\kappa})T_0}C(M,\nu,\bar{\kappa})T_0\leq1,
\end{align}
where $C(M,\nu,\bar{\kappa})$ is given in the proof of Lemma \ref{yl3.1} below.

\begin{lem}\label{yl3.1} Let ${(\sigma^k,\ u^k)} $ be a sequence of the approximate solutions generates by
$(\ref{eqnim1})-(\ref{eqnim2}) $ together with the initial step $(\sigma^0,u^0)=(\sigma_0, u_0)$. Then the following estimate holds
\begin{equation}
\underset{{0\leq t\leq T_0}}{\sup}\|\sigma^k\|_{H^s}+\underset{{0\leq t\leq T_0}}{\sup}\|u^k\|_{H^{s}}\leq M,\quad for\  all\  k\geq 0,
\end{equation}
where $s>\frac{N}{2}+1$, $ M$ and $T_0$ are given in $(\ref{ML})$ and $ (\ref{T0})$.
\end{lem}
\begin{proof}
 We use the method of induction to prove the Lemma.\\

{\bf Step 1.} (Initial step) Because we choose $(\sigma^0,u^0)=(\sigma_0,u_0)$, together with the choice of $M,\ T_0>0$ in $(\ref{ML})$ and $ (\ref{T0})$, it is easy to check that
$$
\underset{{0\leq t\leq T_0}}{\sup}(\|\sigma^0\|_{H^{s}}+\|u^0\|_{H^{s}})\leq M.
$$

%{\bf Step 2.} (inductive step)
{\bf Step 2.} (Inductive step) Suppose that
\begin{equation}\label{ukm}
\underset{{0\leq t\leq T_0}}{\sup}(\|\sigma^k\|_{H^{s}}+\|u^k\|_{H^{s}})\leq M,
\end{equation}
where $T_0,\ M$ are positive constants determined in $(\ref{ML})$ and $ (\ref{T0})$.
We will prove that
\begin{align*}
\underset{{0\leq t\leq T_0}}{\sup}(\|\sigma^{k+1}\|_{H^s}+\|u^{k+1}\|_{H^s})\leq M.
\end{align*}
%{\bf Step 2.1} (Estimate of $L^2$-norm)
First, multiplying $\sigma^{k+1}, \ u^{k+1}$ on both sides of (\ref{eqnim1}), (\ref{eqnim2}) respectively, summing up and integrating over $\mathbb{R}^N$, we obtain
\begin{align}\label{L21}
&\frac{1}{2}\frac{d}{dt}\Big(\|\sigma^{k+1}\|_{L^2}^2+\|u^{k+1}\|_{L^2}^2\Big)+\frac{1}{\tau}\|u^{k+1}\|_{L^2}^2\nonumber\\
&\quad=-\int \Big(u^k\cdot\nabla \sigma^{k+1} \sigma^{k+1}+ u^k\cdot\nabla u^{k+1}\cdot u^{k+1}\Big)\mathrm{d}x%\nonumber\\
%&\quad
-\frac{1}{\nu}\int \Big(\sigma^k\nabla\cdot u^{k+1}\sigma^{k+1}+\sigma^k\nabla \sigma^{k+1}\cdot u^{k+1}\Big)\mathrm{d}x\nonumber\\
&\quad\quad-a\int \int \mathbf \Gamma(x-y)(u^k(x)-u^k(y))
\Big(\frac{1}{\nu} \sigma^k(y)+\bar{\kappa}\Big)^\nu \mathrm{d}y\cdot u^{k+1}(x)\mathrm{d}x\nonumber\\
&\quad=I_1+I_2+I_3.
\end{align}

We shall estimate the terms on the right-hand side of (\ref{L21}). Thanks to  the Sobolev embedding theorem and the inductive assumption (\ref{ukm}), using integration by parts, we obtain

\begin{align}
&I_1=-\int \Big(u^k\cdot\nabla \sigma^{k+1} \sigma^{k+1}+ u^k\cdot\nabla u^{k+1}\cdot u^{k+1}\Big)\mathrm{d}x
=\frac{1}{2}\int \nabla\cdot u^{k}\Big(|\sigma^{k+1}|^2+|u^{k+1}|^2\Big)\mathrm{d}x\nonumber\\
&\quad\leq \|\nabla\cdot u^{k}\|_{L_{\infty}}\Big(\|\sigma^{k+1}\|_{L^2}^2+\|u^{k+1}\|_{L^2}^2\Big)\leq CM\Big(\|\sigma^{k+1}\|_{L^2}^2+\|u^{k+1}\|_{L^2}^2\Big);\nonumber\\
&I_2=-\frac{1}{\nu}\int \Big(\sigma^k\nabla\cdot u^{k+1}\sigma^{k+1}+\sigma^k\nabla \sigma^{k+1}\cdot u^{k+1}\Big)\mathrm{d}x=\frac{1}{\nu}\int \nabla \sigma^k\cdot (u^{k+1}\sigma^{k+1})\mathrm{d}x\nonumber\\
&\quad\leq C\|\nabla \sigma^k\|_{L^{\infty}}\|\sigma^{k+1}\|_{L^2}\|u^{k+1}\|_{L^2}\leq CM\Big(\|\sigma^{k+1}\|_{L^2}^2+\|u^{k+1}\|_{L^2}^2\Big);\nonumber\\
&I_3=-a\int \int \mathbf \Gamma(x-y)(u^k(x)-u^k(y))
\Big(\frac{1}{\nu} \sigma^k(y)+\bar{\kappa}\Big)^\nu\cdot u^{k+1}(x)\mathrm{d}y\mathrm{d}x\nonumber\\
&\quad\leq C\|(\frac{1}{\nu} \sigma^k+\bar{\kappa})^\nu\|_{L^{\infty}}\|\mathbf \Gamma\|_{L^1}\|u^k\|_{L^2}\|u^{k+1}\|_{L^2}
\nonumber\\
&\quad\leq C(M,\nu,\bar{\kappa})(\|u^{k+1}\|_{L^2}^2+1).\nonumber
\end{align}
%Here, the estimate of $I_3$ thanks to (\ref{gamma1}).
Combining the estimate of $I_i,\ i=1,\ 2,\ 3$, we obtain
\begin{equation}\label{L2un}
\frac{d}{dt}\Big(\|\sigma^{k+1}\|_{L^2}^2+\|u^{k+1}\|_{L^2}^2\Big)+\frac{2}{\tau}\|u^{k+1}\|_{L^2}^2\leq C(M,\nu,\bar{\kappa})\Big(\|\sigma^{k+1}\|_{L^2}^2+\|u^{k+1}\|_{L^2}^2+1\Big).
\end{equation}

%{\bf Step 2.2} (Estimate of $H^s$-norm)
Next we will get the higher order estimate of $(\sigma^{k+1}, u^{k+1})$.

Taking $\nabla^{r}, 1\leq r\leq s$ with respect to $x$ on both sides of $(\ref{eqnim1})$-$(\ref{eqnim2})$, and then multiplying  the resulting identities by $\nabla^r \sigma^{k+1},\ \nabla^r u^{k+1}$ respectively, summing up and integrating over $\mathbb{R}^N$, we obtain
\begin{align}\label{L2}
&\frac{1}{2}\frac{d}{dt}\Big(\|\nabla^r\sigma^{k+1}\|_{L^2}^2+\|\nabla^ru^{k+1}\|_{L^2}^2\Big)+\frac{1}{\tau}\|\nabla^ru^{k+1}\|_{L^2}^2\nonumber\\
&\quad=-\int\Big(\nabla^r(u^k\cdot\nabla \sigma^{k+1})\nabla^r\sigma^{k+1}+ \nabla^r(u^k\cdot\nabla u^{k+1})\cdot \nabla^ru^{k+1}\Big)\mathrm{d}x\nonumber\\
&\quad-\frac{1}{\nu}\int\Big( \nabla^r(\sigma^k\nabla\cdot u^{k+1})\nabla^r \sigma^{k+1}+\nabla^r(\sigma^k\nabla \sigma^{k+1})\cdot \nabla^ru^{k+1}\Big)\mathrm{d}x\\
&\quad-a\int \nabla^r_x\Big(\int \mathbf \Gamma(x-y)u^k(x)
(\frac{1}{\nu} \sigma^k(y)+\bar{\kappa})^\nu dy\Big)\cdot \nabla^ru^{k+1}(x)\mathrm{d}x\nonumber\\
&\quad+a\int \nabla^r_x\Big(\int \mathbf \Gamma(x-y)u^k(y)
(\frac{1}{\nu} \sigma^k(y)+\bar{\kappa})^\nu dy\Big)\cdot \nabla^ru^{k+1}(x)\mathrm{d}x\nonumber\\
&\quad=\sum_{i=1}^{4}I_i.
\end{align}
In the following we will estimate $I_i$ term by term. Using the Sobolev embedding theorem and Moser type inequality, we obtain
\begin{align}
{I_1}&=-\int\Big(\nabla^r(u^k\cdot\nabla \sigma^{k+1})\nabla^r\sigma^{k+1}+ \nabla^r(u^k\cdot\nabla u^{k+1} )\cdot\nabla^ru^{k+1}\Big)\mathrm{d}x\nonumber\\
&=-\int u^k\cdot\nabla\nabla^r\sigma^{k+1}\nabla^r\sigma^{k+1}dx- \int \Big(\nabla^r(u^k\cdot\nabla \sigma^{k+1})-u^k\cdot\nabla\nabla^r\sigma^{k+1}\Big)\nabla^r\sigma^{k+1}dx\nonumber\\
&\quad-\int u^k\cdot\nabla\nabla^ru^{k+1}\nabla^ru^{k+1}dx- \int \Big(\nabla^r(u^k\cdot\nabla u^{k+1})-u^k\cdot\nabla\nabla^ru^{k+1}\Big)\nabla^ru^{k+1}dx\nonumber\\
&\leq C\|\nabla\cdot u^{k}\|_{L^{\infty}}(\|\nabla^r\sigma^{k+1}\|_{L^2}^2+\|\nabla^ru^{k+1}\|_{L^2}^2)\nonumber\\
&\quad+\|\nabla^r\sigma^{k+1}\|_{L^2}(\|\nabla^ru^k\|_{L^2}\|\nabla \sigma^{k+1}\|_{L^{\infty}}+\|\nabla^r\sigma^{k+1}\|_{L^2}\|\nabla u^k\|_{L^{\infty}})\nonumber\\
&\quad+\|\nabla^ru^{k+1}\|_{L^2}(\|\nabla^ru^k\|_{L^2}\|\nabla u^{k+1}\|_{L^{\infty}}+\|\nabla^ru^{k+1}\|_{L^2}\|\nabla u^k\|_{L^{\infty}})\nonumber\\
&\leq CM(\|\nabla^r\sigma^{k+1}\|_{L^2}^2+\|\nabla^ru^{k+1}\|_{L^2}^2)+CM(\|\nabla \sigma^{k+1}\|_{H^{s-1}}^2+\|\nabla u^{k+1}\|_{H^{s-1}}^2);\label{bian1}\\
{I_2}&=-\frac{1}{\nu}\int \Big(\nabla^r(\sigma^k\nabla\cdot u^{k+1})\nabla^r \sigma^{k+1}+\nabla^r(\sigma^k\nabla \sigma^{k+1})\cdot \nabla^ru^{k+1}\Big)\mathrm{d}x\nonumber\\
&=-\frac{1}{\nu}\int \sigma^k\nabla\cdot\nabla^r u^{k+1}\nabla^r \sigma^{k+1}\mathrm{d}x-\frac{1}{\nu}\int_{\mathbb{R}^N}\Big(\nabla^r(\sigma^k\nabla\cdot u^{k+1}) -\sigma^k\nabla\cdot\nabla^r u^{k+1}\Big)\nabla^r \sigma^{k+1}\mathrm{d}x\nonumber\\
&\quad-\frac{1}{\nu}\int \sigma^k\nabla\nabla^r \sigma^{k+1}\nabla^r u^{k+1}\mathrm{d}x-\frac{1}{\nu}\int\Big(\nabla^r(\sigma^k\nabla \sigma^{k+1})-\sigma^k\nabla\nabla^r \sigma^{k+1}\Big)\nabla^r u^{k+1}\mathrm{d}x\nonumber\\
&\leq C\|\nabla \sigma^k\|_{L^{\infty}}(\|\nabla^r\sigma^{k+1}\|_{L^2}^2+\|\nabla^ru^{k+1}\|_{L^2}^2)\nonumber\\
&\quad+C\Big(\|\nabla^r \sigma^k\|_{L^2}\|\nabla\cdot u^{k+1}\|_{L^{\infty}}+\|\nabla^r u^{k+1}\|_{L^2}\|\nabla \sigma^k\|_{L^{\infty}}\Big)\|\nabla^r\sigma^{k+1}\|_{L^2}\nonumber\\
&\quad+C\Big(\|\nabla^r \sigma^k\|_{L^2}\|\nabla \sigma^{k+1}\|_{L^{\infty}}+\|\nabla^r \sigma^{k+1}\|_{L^2}\|\nabla \sigma^k\|_{L^{\infty}}\Big)\|\nabla^ru^{k+1}\|_{L^2}\nonumber\\
&\leq CM(\|\nabla^r\sigma^{k+1}\|_{L^2}^2+\|\nabla^ru^{k+1}\|_{L^2}^2)+CM(\|\nabla \sigma^{k+1}\|_{H^{s-1}}^2+\|\nabla u^{k+1}\|_{H^{s-1}}^2),\label{bian2}
\end{align}
where we have used $\|\nabla u^{k+1}\|_{L^{\infty}}\leq\|\nabla u^{k+1}\|_{H^{s-1}}$  and the inductive assumption (\ref{ukm}). % Applying h$\ddot{o}$lder inequality, we get
%$$
%{I_3}=-\frac{1}{\tau}\int\nabla^ru^k\cdot \nabla^ru^{k+1}\mathrm{d}x\leq 5C\|u^k\|_{L^2}\|u^{k+1}\|_{L^2}\leq CM+CM\|\nabla^ru^{k+1}\|_{L^2}^2;
%$$

Next, we estimate the $I_3$. Using Young's inequality and Moser type inequality, we have
\begin{align}\label{3.14}
I_3&=-a\int \nabla_x^r\Big( u^k(x)\int \mathbf \Gamma(x-y)
(\frac{1}{\nu} \sigma^k(y)+\bar{\kappa})^\nu dy\Big)\cdot \nabla^ru^{k+1}(x)\mathrm{d}x\nonumber\\
&\leq C\|\nabla^ru^{k+1}\|_{L^2}\|\nabla^r\Big( u^k\mathbf \Gamma\ast
(\frac{1}{\nu} \sigma^k+\bar{\kappa})^\nu\Big)\|_{L^2}\nonumber\\
&\leq C\|\nabla^ru^{k+1}\|_{L^2}\|\nabla^r u^k\|_{L^2}\|\mathbf \Gamma\ast
(\frac{1}{\nu} \sigma^k+\bar{\kappa})^\nu\|_{L^\infty}+C\|\nabla^ru^{k+1}\|_{L^2}\| u^k\|_{L^\infty}\|\nabla^r\mathbf \Gamma\ast
(\frac{1}{\nu} \sigma^k+\bar{\kappa})^\nu\|_{L^2}.
\end{align}
Applying the Sobolev embedding theorem and the inductive assumption (\ref{ukm}), direct calculation shows
\begin{align}\label{4.14}
\|\mathbf \Gamma \ast (\frac{1}{\nu}\sigma^k+\bar{\kappa})^\nu\|_{L^\infty}&=\|\int \mathbf \Gamma(x-y)(\frac{1}{\nu}\sigma^k(y)+\bar{\kappa})^\nu\mathrm{d}y\|_{L^\infty}\leq\|\mathbf \Gamma\|_{L^1}\|(\frac{1}{\nu} \sigma^k+\bar{\kappa})^\nu\|_{L^\infty}\nonumber\\
&\leq C(\|\sigma^k\|_{L^\infty},\nu,\bar{\kappa})\leq C(M,\nu,\bar{\kappa}).
\end{align}
\begin{align}
\|\nabla^r \mathbf \Gamma \ast (\frac{1}{\nu}{\nu}\sigma^k+\bar{\kappa})^\nu\|_{L^2}&=\|\nabla^r_x\int \mathbf \Gamma(x-y)(\frac{1}{\nu}\sigma^k(y)+\bar{\kappa})^\nu\mathrm{d}y\|_{L^2}=\|\int \mathbf \Gamma(x-y)\nabla_y^r(\frac{1}{\nu}\sigma^k(y)+\bar{\kappa})^\nu\mathrm{d}y\|_{L^2}\nonumber\\
&\leq C\|\mathbf \Gamma\|_{L^1}\|\nabla^r(\frac{1}{\nu} \sigma^k+\bar{\kappa})^\nu \|_{L^2} \leq C(\|\sigma^k\|_{L^\infty},\nu,\bar{\kappa})\|\mathbf \Gamma\|_{L^1}\|\sigma^k\|_{H^{r}}\nonumber\\
&\leq C(M,\nu,\bar{\kappa}),\label{4.16}
\end{align}
where $C(M,\nu,\bar{\kappa})$  is non-decreasing in $M$.
%\begin{align}\label{nablanu}
%\|\nabla_x^r(\frac{1}{\nu} \sigma^k(y)+\bar{\kappa})^\nu \|_{L^2} \leq C(M,\nu,\bar{\kappa})\|\sigma^k\|_{H^r},
%\end{align}
%for $1\leq r\leq s$.

Then, we obtain that
\begin{equation}
I_3\leq C(M,\nu,\bar{\kappa})(\|\nabla^ru^{k+1}\|^2_{L^2}+1).
\end{equation}
Finally, we provide the estimate of $I_4$. By applying the Moser type inequality and Young's inequality, we have
 \begin{align}
 I_4&=a\int \nabla^ru^{k+1}(x)\mathrm{d}x\nabla_x^r\Big(\int \mathbf \Gamma(x-y)u^k(y)
(\frac{1}{\nu} \sigma^k(y)+\bar{\kappa})^\nu dy\Big)\nonumber\\
&=a\int \nabla^ru^{k+1}(x)\mathrm{d}x\int \mathbf \Gamma(x-y)\nabla_y^r \Big(u^k(y)
(\frac{1}{\nu} \sigma^k(y)+\bar{\kappa})^\nu \Big)dy\nonumber\\
&\leq C\|\mathbf \Gamma\|_{L^1}\|\nabla^ru^{k+1}\|_{L^2}\|\nabla^r \Big(u^k(y)
(\frac{1}{\nu} \sigma^k(y)+\bar{\kappa})^\nu \|_{L^2}\nonumber\\
&\leq C\|\mathbf \Gamma\|_{L^1}\|\nabla^ru^{k+1}\|_{L^2}\Big(\|\nabla^ru^{k}\|_{L^2}\|(\frac{1}{\nu} \sigma^k+\bar{\kappa})^\nu\|_{L^\infty}+\|\nabla^r(\frac{1}{\nu} \sigma^k+\bar{\kappa})^\nu\|_{L^2}\|u^{k}\|_{L^\infty}\Big)\nonumber\\
&\leq C(M,\nu,\bar{\kappa})(\|\nabla^ru^{k+1}\|_{L^2}^2+1).
 \end{align}
Here, we used (\ref{4.16}) and the inductive assumption (\ref{ukm}).

Collecting all estimates of $I_i$ from 1 to 4, we obtain that
\begin{align}\label{Hsun1}
\frac{d}{dt}&\Big(\|\nabla^{r}\sigma^{k+1}\|_{L^2}^2+\|\nabla^{r}u^{k+1}\|_{L^2}^2\Big)+\frac{2}{\tau}\|\nabla^ru^{k+1}\|_{L^2}^2\nonumber\\
&\leq C(M,\nu,\bar{\kappa})(\|\nabla^ru^{k+1}\|_{L^2}^2+\|\nabla^r\sigma^{k+1}\|_{L^2}^2+1)+CM(\|\nabla u^{k+1}\|_{H^{s-1}}^2+\|\nabla\sigma^{k+1}\|_{H^{s-1}}^2).
\end{align}
We can sum (\ref{Hsun1}) over $1\leq r\leq s$ and combine  (\ref{L2un}) to obtain
\begin{align}\label{Hsun}
\frac{d}{dt}\Big(\|\sigma^{k+1}\|_{H^s}^2+\|u^{k+1}\|_{H^s}^2\Big)+\frac{2}{\tau}\|\nabla^ru^{k+1}\|_{L^2}^2&\leq C(M,\nu,\bar{\kappa})(\|u^{k+1}\|_{H^s}^2+\|\sigma^{k+1}\|_{H^s}^2)+C(M,\nu,\bar{\kappa}).
\end{align}

This yields
\begin{align*}
\underset{a\leq t\leq T_0}{sup}\|\sigma^{k+1}\|_{H^s}^2&+\|u^{k+1}\|_{H^s}^2+\int_{0}^{T_0}\|\nabla^ru^{k+1}\|_{L^2}^2dt\nonumber\\
&\leq e^{C(M,\nu,\bar{\kappa})T_0}\Big(\|\sigma_0\|_{H^s}^2+\|u_0\|_{H^s}^2\Big)+e^{C(M,\nu,\bar{\kappa})T_0}C(M,\nu,\bar{\kappa})T_0.
\end{align*}
By the choise of $M$ and $T_0$ as in $(\ref{ML})$ and $ (\ref{T0})$, we can easily check that
$$
e^{C(M,\nu,\bar{\kappa})T_0}\Big(\|\sigma_0\|_{H^s}^2+\|u_0\|_{H^s}^2\Big)+e^{C(M,\nu,\bar{\kappa})T_0}C(M,\nu,\bar{\kappa})T_0\leq M^2.
$$
So, we obtain
\begin{equation}\label{ul2}
\|\sigma^{k+1}\|_{H^s}+\|u^{k+1}\|_{H^s}\leq M
\end{equation}
which completes the induction process.
\end{proof}
%\section{Existence an uniqueness of local classical solution}
%In this part, we will show the convergence of approximate solution $\{\sigma^k,\ u^k\}$. Based on this, the existence of the solutions can be directly obtained by taking the limit $k\rightarrow \infty$ in \eqref{eqn1la}-\eqref{eqn2la}. Although we cannot directly show that the sequence of the
%approximate solutions is convergent in the Sobolev space of the desired order by
%loosely following the argument in [Ref. 28], we can show that the limit function
%obtained in a lower-order Sobolev space belongs to the desired solution space.

\subsection{Convergence in lower-order norm} In this subsection, we will show that the  $\{\sigma^k,\ u^k\}_{k=1}^{\infty}$ are convergent in some lower-order Sobolev spaces using the contraction mapping principle.
%\begin{lem}
%There exists $n<1$ such that
%\begin{eqnarray*}
%&&\sup_{t\in [0, T_*]}\{\|\rho^{k+1}-\rho^k\|_{L^2}^2+\|u^{k+1}-u^k\|_{L^2}^2\}\nonumber\\&&\leq n\sup_{t\in [0, %T_*]}\{\|\rho^k-\rho^{k-1}\|_{L^2}^2+\|u^k-u^{k-1}\|_{L^2}^2\}+C2^{-k}
%\end{eqnarray*}
%holds for every $k$.
%\end{lem}
%\begin{proof}

Let $$\overline{n}^{k+1}=\sigma^{k+1}-\sigma^{k},\ \overline{u}^{k+1}=u^{k+1}-u^{k}.$$ Note that $(\sigma^{k+1},\ u^{k+1})$ and $(\rho^{k},\ u^{k})$ satisfies
\begin{align}
 \partial_t \sigma^{k+1}&+\bar{\kappa}\nabla\cdot u^{k+1} =-u^{k}\cdot \nabla \sigma^{k+1}-\dfrac{1}{\nu}\sigma^{k}\nabla\cdot u^{k+1},\nonumber\\
\partial_t u^{k+1}&+\bar{\kappa}\nabla \sigma^{k+1}+\frac{1}{\tau} u^{k+1}=-u^{k}\cdot \nabla u^{k+1}-\dfrac{1}{\nu}\sigma^{k}\nabla \sigma^{k+1}\label{k+1}\\
&-a\int\mathbf \Gamma(x-y)(u^k(x)-u^k(y))
(\frac{1}{\nu} \sigma^k(y)+\bar{\kappa})^\nu dy\nonumber
\end{align}
as well as
\begin{align}
\partial_t \sigma^{k}&+\bar{\kappa}\nabla\cdot u^{k}+\frac{1}{\tau} u^{k}=-u^{k-1}\cdot \nabla \sigma^{k}-\dfrac{1}{\nu}\sigma^{k-1}\nabla\cdot u^{k},\nonumber\\
\partial_t u^{k}&+\bar{\kappa}\nabla \sigma^{k}=-u^{k-1}\cdot \nabla u^{k}-\dfrac{1}{\nu}\sigma^{k-1}\nabla \sigma^{k}\label{k}\\
&-a\int_{\mathbb{R}^N}\mathbf \Gamma(x-y)(u^{k-1}(x)-u^{k-1}(y))
(\frac{1}{\nu} \sigma^{k-1}(y)+\bar{\kappa})^\nu dy\nonumber
\end{align}
subject to the same initial data
\begin{equation}
( \sigma^{k+1},\ u^{k+1})=( \sigma^k,\ u^k)=( n_0,\ u_0)\in\ H^s.
\end{equation}
 It follows from $(\ref{k+1})$ and $(\ref{k})$ that
\begin{align}
\label{nk+1}&\partial_t ( \sigma^{k+1}- \sigma^k)+\bar{\kappa}\nabla\cdot( u^{k+1}- u^k)=-(u^k-u^{k-1})\nabla \sigma^{k+1}\nonumber\\
&\quad\ \quad\quad-u^{k-1}\nabla(\sigma^{k+1}-\sigma^k)-\frac{1}{\nu}(\sigma^k-\sigma^{k-1})\nabla\cdot u^{k+1}-\frac{1}{\nu}\sigma^{k-1}\nabla\cdot (u^{k+1}-u^k)\\
\label{uk+1}&\partial_t ( u^{k+1}- u^k)+\bar{\kappa}\nabla( \sigma^{k+1}- \sigma^k)+\frac{1}{\tau}(u^{k+1}-u^{k})\nonumber\\
&=-(u^k-u^{k-1})\nabla\cdot u^{k+1}-u^{k-1}\nabla\cdot(u^{k+1}-u^k)
-\frac{1}{\nu}(\sigma^k-\sigma^{k-1})\nabla \sigma^{k+1}-\frac{1}{\nu}\sigma^{k-1}\nabla (\sigma^{k+1}-\sigma^k)\nonumber\\
&\quad-a\int \mathbf \Gamma(x-y)(u^k(x)-u^k(y))\Big((\frac{1}{\nu} \sigma^{k}(y)+\bar{\kappa})^\nu-(\frac{1}{\nu} \sigma^{k-1}(y)+\bar{\kappa})^\nu\Big)dy\nonumber\\
&\quad-a\int \mathbf \Gamma(x-y)\Big(u^k(x)-u^{k-1}(x)\Big)(\frac{1}{\nu} \sigma^{k-1}(y)+\bar{\kappa})^\nu dy\nonumber\\
&\quad-a\int \mathbf \Gamma(x-y)\Big(u^k(y)-u^{k-1}(y)\Big)(\frac{1}{\nu} \sigma^{k-1}(y)+\bar{\kappa})^\nu dy.
\end{align}
Multiplying (\ref{nk+1}) and  (\ref{uk+1}) by $( \sigma^{k+1}- \sigma^k)$, $( u^{k+1}- u^k)$ respectively, summing up and integrating over $\mathbb{R}^N$, similar to the estimate in subsection 3.2, we obtain
\begin{align}\label{ct}
&\frac{d}{dt}\Big(| \sigma^{k+1}- \sigma^k|_{L^2}^2+| u^{k+1}- u^k|_{L^2}^2\Big)+\frac{2}{\tau}\|u^{k+1}- u^k\||_{L^2}^2\nonumber\\
&=-\int (u^k-u^{k-1})\nabla \sigma^{k+1}(\sigma^{k+1}-\sigma^k)+(u^k-u^{k-1})\nabla\cdot u^{k+1}\cdot(u^{k+1}-u^k)\mathrm{d}x\nonumber\\
&\quad-\int u^{k-1}\nabla (\sigma^{k+1}-\sigma^k)(\sigma^{k+1}-\sigma^k)+u^{k-1}\nabla\cdot (u^{k+1}-u^k)\cdot(u^{k+1}-u^k)\mathrm{d}x\nonumber\\
&\quad-\frac{1}{\nu}\int (\sigma^{k}-\sigma^{k-1})\nabla \cdot u^{k+1}(\sigma^{k+1}-\sigma^k)+(\sigma^{k}-\sigma^{k-1})\nabla \sigma^{k+1}(u^{k+1}-u^k) \mathrm{d}x\nonumber\\ &\quad-\frac{1}{\nu}\int \sigma^{k-1}\nabla \cdot (u^{k+1}-u^k)(\sigma^{k+1}-\sigma^k)+\sigma^{k-1}\nabla (\sigma^{k+1}-\sigma^k)(u^{k+1}-u^k) \mathrm{d}x\nonumber\\
&\quad-a\int (u^{k+1}-u^k) \mathrm{d}x\int \mathbf \Gamma(x-y)(u^k(x)-u^k(y)\Big((\frac{1}{\nu} \sigma^{k}(y)+\bar{\kappa})^\nu-(\frac{1}{\nu} \sigma^{k-1}(y)+\bar{\kappa})^\nu\Big)\mathrm{d}y\nonumber\\
&\quad-a\int (u^{k+1}-u^k) \mathrm{d}x\int \mathbf \Gamma(x-y)(u^k(x)-u^{k-1}(x))(\frac{1}{\nu} \sigma^{k-1}(y)+\bar{\kappa})^\nu dy\nonumber\\
&\quad-a\int (u^{k+1}-u^k) \mathrm{d}x\int \mathbf \Gamma(x-y)(u^k(y)-u^{k-1}(y))(\frac{1}{\nu} \sigma^{k-1}(y)+\bar{\kappa})^\nu dy\nonumber\\
&\leq C(M,\nu,\bar{\kappa})\Big(\| \sigma^{k+1}- \sigma^k\|_{L^2}^2+\| u^{k+1}- u^k\|_{L^2}^2\Big)+C(M,\nu,\bar{\kappa})\Big(\| \sigma^{k}- \sigma^{k-1}\|_{L^2}^2+\| u^{k}- u^{k-1}\|_{L^2}^2\Big),
\end{align}
where we use the following estimate, for $a.e.(t,x)\in (0,+\infty)\times \mathbb{R}^N$
\begin{align}\label{mean}
\Big|(\frac{1}{\nu} \sigma^{k}+\bar{\kappa})^\nu-(\frac{1}{\nu} \sigma^{k-1}+\bar{\kappa})^\nu\Big|&=\Big|\int_{0}^{1}\Big(\frac{s}{\nu}\sigma^k+\frac{(1-s)}{\nu}\sigma^{k-1}+\bar{\kappa}\Big)^{\nu-1}(\sigma^k-\sigma^{k-1})\mathrm{d}s\Big|\nonumber\\
&\leq C(M,\nu,\bar{\kappa})|\sigma^k-\sigma^{k-1}|.
\end{align}

We can integrate (\ref{ct}) over $(0,\ t)$ to obtain
\begin{align}\label{uh}
\underset{0\leq\tilde{t}\leq t}{\sup}\|\sigma^{k+1}- \sigma^k\|_{L^2}^2+\| u^{k+1}- u^k\|_{L^2}^2&\leq C\int_{0}^{t}\|\sigma^{k+1}(\tilde{t})- \sigma^k(\tilde{t})\|_{L^2}^2+\| u^{k+1}(\tilde{t})- u^k(\tilde{t})\|_{L^2}^2 \mathrm{d}\tilde{t}\nonumber\\
&+C\int_{0}^{t}\| \sigma^{k}(\tilde{t})- \sigma^{k-1}(\tilde{t})\|_{L^2}^2+\| u^{k}(\tilde{t})- u^{k-1}(\tilde{t})\|_{L^2}^2 \mathrm{d}\tilde{t}.
\end{align}

The we sum up for $k=1,\ 2,\ \cdots$ together with the Gronwall's inequality to obtain
\begin{align}\label{conver}
\sum_{k=1}^{\infty}\|\sigma^{k+1}- \sigma^k\|_{L^2}^2+\| u^{k+1}- u^k\|_{L^2}^2\leq C, \quad for \ t\leq T_0.
\end{align}
This implies that ${ \sigma^k}$ and ${u^k}$ are Cauchy sequences in $C([0,
 T_0];\ L^2)$ .
\subsection{The proof of Theorem \ref{local} }In this subsection, we will prove the local well-posedness of the system (\ref{eqn1a})-(\ref{eqn2a}) given in Theorem 2.1. First, we prove the existence of classical solutions.

By the Gagliardo-Nirenberg inequality together with the uniform bound of $( \sigma^k,\ u^k)$ and  the convergence result (\ref{conver}), we can conclude that, for $s>\frac{N}{2}+1$,
\begin{align}\label{conv}
 \sigma^{k}\rightarrow \sigma\in C([0,\ T_0];\ H^{s-1}) \ and \ u^{k}\rightarrow u\in C([0,\ T_0];\ H^{s-1}).
\end{align}
It easily follows from (\ref{conv}) that limit function $( \sigma,\ u)$ is a solution to (\ref{eqn1a})-(\ref{eqn2a}) in a distributional sense.
Using a similar argument as in \cite{HKK}, we can obtain the regularity of $( \sigma,\ u)$:
\begin{align}\label{conv1}
 (\sigma,u)\in C([0,\ T_0];\ H^s).
\end{align}
Applying Sobolev's embedding theorem, we prove $( \sigma,\ u )\in C^{1}([0,\ T_0]\times \mathbb{R}^N)$ is a classical solution.

Next, we prove the uniqueness. Let $( \sigma,\ u)$ and $(\widetilde{ \sigma},\ \widetilde{u})$
be the two classical solutions of $(\ref{eqn1a})-(\ref{eqn2a})$ corresponding the same initial data $(\sigma_0,\ u_0)$. We set
$$
U(t)=\|\sigma-\widetilde{ \sigma}\|_{L^2}^2+\|u-\widetilde{u}\|_{L^2}^2.
$$
Then, by the same argument as in subsection 4.1, $U(t)$ satisfies Gronwall's inequality:
$$
U(t)\leq C\int_{0}^tU(t)dt,\quad U(0)=0.
$$
This yields that
$$
 \sigma\equiv\widetilde{ \sigma},\ u\equiv\widetilde{u}\in C([0,\ T];\ L^{2}( \mathbb{R}^N)).
$$
So, we complete the proof of theorem.
\section{Global existence of classical solution}
In this section, we discuss the global existence of the classical solution  on the basis of the local existence results in Section 3. According to Remark \ref{rem}, we assume that the background density and the bottom viscous damping satisfy

%the actual situation, in order to obtain the global solution, we need to assume that the background density is less than the bottom viscous damping, that is, choose $\bar{\kappa}$ satisfying
\begin{equation}\label{mu}
2a\bar{\kappa}^\nu\|\mathbf \Gamma\|_{L^1}<\frac{1}{\tau}.
\end{equation}

\subsection{A priori esimates}
In this subsection, we will provide the a priori estimates for  the  Cauchy problem (\ref{eqn1a})-(\ref{eqn3a}).
%For notional simplicity, we denote by $U:=(n,\ u),$ i.e. $U_0=(n_0,\ u_0)$ and $\|U\|_{H^s}=\|(n,\ u)\|_{H^s},$ for $s\geq 0$ in the rest of this section.
Hence, we assume a priori assumption that for $s>\frac{N}{2}$ and a sufficiently small $\delta>0$,
\begin{equation}\label{xiao}
\underset{{0\leq t\leq T_0}}{\sup}\Big(\|\sigma\|_{H^s}+\|u\|_{H^s} \Big)\leq \delta.
\end{equation}

%\begin{align}

%&\label{eqn1a}\partial_t n+\bar{\kappa}\nabla\cdot u =-u\cdot \nabla n-\frac{1}{2}n\nabla\cdot u,\\
%&\label{eqn2a}\partial_t u+\bar{\kappa}\nabla n=-u\cdot \nabla u-\frac{1}{2}n\nabla\cdot n-\frac{1}{\tau} u-\frac{1}{2}\int_{\mathbb{R}^N}\mathbf \Gamma(x-y)(u(x)-u(y))(\frac{1}{2}n(y)+\bar{\kappa})^2dy,
%\end{align}

We show the $L^2$-norm estimates which contains the dissipation estimate for $u$. It should be noticed that there is no dissipation estimate of $\|\sigma\|_{L^2}$.

\begin{lem}\label{5.1}Assume $(\sigma,u)$ are classical solution of $(\ref{eqn1a})-(\ref{eqn2a})$ and $(\ref{mu}), (\ref{xiao})$ hold, then we have
\begin{equation}\label{gul2}
\frac{1}{2}\frac{d}{dt}\Big(\|\sigma\|_{L^2}^2+\|u\|_{L^2}^2\Big)+(\frac{1}{\tau}-2a\bar{\kappa}^\nu\|\mathbf \Gamma\|_{L^1})\|u\|_{L^2}^2\leq C\delta (\|u\|_{L^2}^2+\|\nabla \sigma\|_{L^2}^2).
\end{equation}
\end{lem}
\begin{proof}
We multiply (\ref{eqn1a}) and (\ref{eqn2a}) by $\sigma$ and $u$ respectively, sum up and integrate over $\mathbb{R}^N$, we obtain
\begin{align}\label{gul}
\frac{1}{2}\frac{d}{dt}\int \Big(|\sigma|^2&+|u^2|\Big)\mathrm{d}x+\frac{1}{\tau}\|u\|_{L^2}^2\nonumber\\
&=-\int \Big(u\cdot \nabla \sigma\sigma+u\cdot \nabla u\cdot u\Big)\mathrm{d}x-\frac{1}{\nu}\int \Big(\sigma\nabla\cdot u \sigma+\sigma\nabla \sigma\cdot u\Big)\mathrm{d}x\nonumber\\
&\quad-a\int \int \mathbf \Gamma(x-y)(u(x)-u(y))(\frac{1}{\nu}\sigma(y)+\bar{\kappa})^\nu\mathrm{d}y\cdot u(x)\mathrm{d}x\nonumber\\
&=I_1+I_2+I_3.
\end{align}
We estimate $I_i$ item by item. Using Young's inequality, we have
\begin{align}
I_1&=-\int \Big(u\cdot \nabla \sigma \sigma+u\cdot \nabla u\cdot u\Big)\mathrm{d}x=-\int u\cdot \nabla \sigma \sigma\mathrm{d}x+\frac{1}{2}\int |u|^2\nabla\cdot u\mathrm{d}x\nonumber\\
&\leq C\|\sigma\|_{L^\infty}\|u\|_{L^2}\|\nabla \sigma\|_{L^2}+C\|\nabla\cdot u\|_{L^\infty}\|u\|^2_{L^2}\nonumber\\
&\leq C\delta(\|\nabla \sigma\|_{L^2}^2+\|u\|^2_{L^2}),\label{1111}\\
I_2&=-\frac{1}{\nu}\int \Big(\sigma\nabla\cdot u \sigma+\sigma\nabla \sigma\cdot u\Big)\mathrm{d}x=\frac{1}{\nu}\int \sigma u\cdot \nabla \sigma\mathrm{d}x
\leq C\|\sigma\|_{L^\infty}\|u\|_{L^2}\|\nabla \sigma\|_{L^2}\nonumber\\
&\leq C\delta(\|\nabla \sigma\|_{L^2}^2+\|u\|^2_{L^2});\\
I_3&=-a\int u(x)\mathrm{d}x \int \mathbf \Gamma(x-y)(u(x)-u(y))(\frac{1}{\nu}\sigma(y)+\bar{\kappa})^\nu\mathrm{d}y \nonumber\\
&=-a\int u^2(x)\int \mathbf \Gamma(x-y)\Big((\frac{1}{\nu}\sigma(y)+\bar{\kappa})^\nu-\bar{\kappa}^\nu\Big)\mathrm{d}y\mathrm{d}x \nonumber\\
&\quad +a\int u(x)\int \mathbf \Gamma(x-y)u(y)\Big((\frac{1}{\nu}\sigma(y)+\bar{\kappa})^\nu-\bar{\kappa}^\nu\Big)\mathrm{d}y\mathrm{d}x \nonumber\\
&\quad-a\bar{\kappa}^\nu\int u(x)\mathrm{d}x \int \mathbf \Gamma(x-y)(u(x)-u(y))\mathrm{d}y\nonumber\\
&\leq 2a\|\mathbf \Gamma\|_{L^1}\|(\frac{1}{\nu}\sigma+\bar{\kappa})^\nu-\bar{\kappa}^\nu\|_{L^\infty}\|u\|_{L^2}^2+2a\kappa^\nu\|\mathbf \Gamma\|_{L^1}\|u\|_{L^2}^2.\nonumber
\end{align}
Similar to (\ref{mean}), we can get
\begin{align}
\|(\frac{1}{\nu}\sigma+\bar{\kappa})^\nu-\bar{\kappa}^\nu\|_{L^\infty}\leq C(\|\sigma\|_{H^{s-1}},\nu,\bar{\kappa}) \|\sigma\|_{L^\infty}\leq C(\|\sigma\|_{H^{s-1}},\nu,\bar{\kappa}) \delta,
\end{align}
with the help of the Sobolev embedding theorem.
Then, we obtain that
\begin{align}
I_3\leq C\delta\|u\|_{L^2}^2+2a\bar{\kappa}^\nu\|\mathbf \Gamma\|_{L^1 }\|u\|_{L^2}^2.\label{222}
\end{align}

Collecting estimates (\ref{1111})-(\ref{222}) into (\ref{gul}), we obtain (\ref{gul2}).
\end{proof}

Next,we provide the high order energy estimates which contains the dissipation estimate for $u$.
\begin{lem}\label{5.2}Assume $1\leq r \leq s$  and $(\ref{mu}), (\ref{xiao})$ hold, then for $s\geq\frac{N}{2}+1$, we have
\begin{equation}\label{guuh}
\frac{1}{2}\frac{d}{dt}\Big(\|\nabla^r\sigma\|_{L^2}^2+\|\nabla^ru\|_{L^2}^2\Big)+(\frac{1}{\tau}-2a\bar{\kappa}^\nu\|\mathbf \Gamma\|_{L^1})\|\nabla^ru\|_{L^2}^2\leq C\delta (\|\nabla^ru\|_{L^2}^2+\|\nabla^r \sigma\|_{L^2}^2)+C\delta \| u\|_{H^{s-1}}^2.
\end{equation}
\end{lem}
\begin{proof}
For $1\leq r\leq s$, we apply $\nabla^r$ to (\ref{eqn1a}), (\ref{eqn2a}), and multiply the resulting identities by $\nabla^r \sigma,\ \nabla^r u$ respectively, sum up and integrating over $\mathbb{R}^N$ to obtain
\begin{align}\label{guh}
\frac{1}{2}\frac{d}{dt}\int\Big( |\nabla^r\sigma|^2&+|\nabla^ru|^2\Big)\mathrm{d}x+\frac{1}{\tau}\|\nabla^ru\|_{L^2}^2\nonumber\\
&=-\int \Big( \nabla^r(u\cdot \nabla \sigma)\nabla^r \sigma+\nabla^r(u\cdot \nabla u)\cdot \nabla^ru\Big)\mathrm{d}x\nonumber\\
&\quad-\frac{1}{\nu}\int\Big( \nabla^r(\sigma\nabla\cdot u)\nabla^r \sigma+\nabla^r(\sigma\nabla \sigma)\cdot \nabla^ru\Big)\mathrm{d}x\nonumber\\
&\quad-a\int \nabla_x^r\Big(\int \mathbf \Gamma(x-y)u(x)(\frac{1}{\nu}\sigma(y)+\bar{\kappa})^\nu\mathrm{d}y\Big)\cdot \nabla^ru(x)\mathrm{d}x\nonumber\\&\quad+a\int \nabla_x^r\Big(\int \mathbf \Gamma(x-y)u(y)(\frac{1}{\nu}\sigma(y)+\bar{\kappa})^\nu\mathrm{d}y\Big)\cdot \nabla^ru(x)\mathrm{d}x\nonumber\\
&=\sum_{i=1}^{4}I_i.
\end{align}
Similar to the estimate of (\ref{bian1}) and (\ref{bian2}) in Section 3, by H${\rm\ddot{o}}$lder's inequality and Moser type inequality, we have
\begin{align}
I_1&=-\int \nabla^r(u\cdot \nabla \sigma)\nabla^r \sigma+\nabla^r(u\cdot \nabla u)\cdot \nabla^ru\mathrm{d}x%\nonumber\\
\leq C\delta \Big(\|\nabla^r\sigma\|_{L^2}^{2}+\|\nabla^ru\|_{L^2}^{2}\Big),\label{i1}\\
I_2&=-\frac{1}{\nu}\int \nabla^r(\sigma\nabla\cdot u)\nabla^r \sigma+\nabla^r(\sigma\nabla \sigma)\cdot \nabla^ru\mathrm{d}x
\leq C\delta \Big(\|\nabla^r\sigma\|_{L^2}^{2}+\|\nabla^ru\|_{L^2}^{2}\Big)\label{i2}.
\end{align}
%with help  Moser-type calculus inequality (\ref{moser}).

Next, we estimate $I_3$. %For clarity, let's make $g(x)=\mathbf \Gamma \ast ({\nu}\sigma(y)+\bar{\kappa})^\nu$, then
Applying Moser type inequality and the H${\rm \ddot{o}}$lder inequality we have
\begin{align}
I_3&=-a\int\nabla^ru(x) \nabla_x^r\Big(\int \mathbf \Gamma(x-y)u(x)(\frac{1}{\nu}\sigma(y)+\bar{\kappa})^\nu\mathrm{d}y\Big)\cdot \mathrm{d}x\nonumber\\
&=-a\int\nabla^ru(x) \nabla^r \Big(u(x)\mathbf \Gamma \ast (\frac{1}{\nu}\sigma+\bar{\kappa})^\nu(x)\Big) \mathrm{d}x\nonumber\\
&=-a\int\nabla^ru(x) \mathbf \Gamma \ast (\frac{1}{\nu}\sigma+\bar{\kappa})^\nu(x)\nabla^r u(x)\mathrm{d}x\nonumber\\
&\quad-a\int\nabla^ru(x)\Big( \nabla^r (u(x)\mathbf \Gamma \ast (\frac{1}{\nu}\sigma+\bar{\kappa})^\nu(x))- \mathbf \Gamma \ast (\frac{1}{\nu}\sigma+\bar{\kappa})^\nu(x)\nabla^r u(x)\Big)\mathrm{d}x\nonumber\\
&\leq a\|\mathbf \Gamma \ast (\frac{1}{\nu}\sigma+\bar{\kappa})^\nu\|_{L^\infty}\|\nabla^ru\|_{L^2}^2\nonumber\\
&\quad+C\|\nabla^ru\|_{L^2}\Big(\|u\|_{L^\infty}\|\nabla^r\mathbf \Gamma \ast (\frac{1}{\nu}\sigma+\bar{\kappa})^\nu\|_{L^2}+\|\nabla \mathbf \Gamma \ast (\frac{1}{\nu}\sigma+\bar{\kappa})^\nu\|_{L^\infty}\|\nabla^{r-1}u\|_{L^2}\Big).
\label{I3z}
\end{align}
To deal with the dissipation of $u$, we need the following estimates. Similar to (\ref{mean}), we can get
\begin{align}\label{4.141}
\|\mathbf \Gamma \ast (\frac{1}{\nu}\sigma+\bar{\kappa})^\nu\|_{L^\infty}&=\|\int \mathbf \Gamma(x-y)(\frac{1}{\nu}\sigma(y)+\bar{\kappa})^\nu\mathrm{d}y\|_{L^\infty}\nonumber\\
&\leq\|\int \mathbf \Gamma(x-y)\Big((\frac{1}{\nu}\sigma+\bar{\kappa})^\nu-\bar{\kappa}^\nu\Big)\mathrm{d}y\|_{L^\infty}+\|\int\bar{\kappa}^\nu \mathbf \Gamma(x-y)\mathrm{d}y\|_{L^\infty}\nonumber\\
&\leq C(\|\sigma\|_{L^{\infty}}, \nu,\bar{\kappa})\|\mathbf \Gamma\|_{L^1}\|\sigma\|_{L^{\infty}}+\bar{\kappa}^\nu\|\mathbf \Gamma\|_{L^1}.
\end{align}
Using the differential properties of the convolution and the Sobolev embedding theorem, we can compute
\begin{align}
\|\nabla \mathbf \Gamma \ast ({\nu}\sigma+\bar{\kappa})^\nu\|_{L^\infty}&=\|\nabla_x\int \mathbf \Gamma(x-y)(\frac{1}{\nu}\sigma(y)+\bar{\kappa})^\nu\mathrm{d}y\|_{L^\infty}=\|\int \mathbf \Gamma(x-y)\nabla_y(\frac{1}{\nu}\sigma(y)+\bar{\kappa})^\nu\mathrm{d}y\|_{L^\infty}\nonumber\\
&\leq\|\mathbf \Gamma\|_{L^1}\|\frac{1}{\nu}\sigma(y)+\bar{\kappa})^{\nu-1}\|_{L^{\infty}}\|\nabla\sigma\|_{L^\infty}\nonumber\\
&\leq C(\|\sigma\|_{L^{\infty}}, \nu,\bar{\kappa})\|\mathbf \Gamma\|_{L^1}\|\nabla \sigma\|_{H^{s-1}},\label{4.15}
\end{align}
and similar to (\ref{4.16}), we have
\begin{align}
\|\nabla^r \mathbf \Gamma \ast ({\nu}\sigma+\bar{\kappa})^\nu\|_{L^2}
\leq C(\|\sigma\|_{L^{\infty}}, \nu,\bar{\kappa})\|\mathbf \Gamma\|_{L^1}\|\sigma\|_{H^{r}},\label{4.161}
\end{align}
where $C(\|\sigma\|_{L^{\infty}}, \nu,\bar{\kappa})$ is non-decreasing in $\|\sigma\|_{L^{\infty}}$.

Substituting (\ref{4.141})-(\ref{4.161}) for (\ref{I3z}), we obtain that
\begin{align}
I_3&\leq C\Big(\|\sigma\|_{L^{\infty}}\|\nabla^ru\|_{L^2}^2+\|\sigma\|_{ H^{r}}\|u\|_{H^{s-1}}\|\nabla^ru\|_{L^2}+\|\nabla \sigma\|_{H^{s-1}}\|\nabla^{r-1}u\|_{L^2}\|\nabla^ru\|_{L^2}\Big)\nonumber\\
&\quad+a\bar{\kappa}^\nu\|\mathbf \Gamma\|_{L^1}\|\nabla^ru\|_{L^2}^2\nonumber\\
&\leq C\delta(\|\nabla^ru\|_{L^2}^2+\|u\|_{H^{s-1}}^2)+a\bar{\kappa}^\nu\|\mathbf \Gamma\|_{L^1}\|\nabla^ru\|_{L^2}^2,\label{I3}
\end{align}
where  the Sobolev embedding theorem is used.

Similar to estimate of $I_3$, we can deduce that
\begin{align}\label{I4z}
I_4&=a\int\nabla^ru(x) \nabla_x^r\Big(\int \mathbf \Gamma(x-y)u(y)(\frac{1}{\nu}\sigma(y)+\bar{\kappa})^\nu\mathrm{d}y\Big) \mathrm{d}x\nonumber\\
&=a\int\nabla^ru(x)\int \mathbf \Gamma(x-y)\nabla_y^r \Big(u(y)(\frac{1}{\nu}\sigma(y)+\bar{\kappa})^\nu\Big)\mathrm{d}y\mathrm{d}x\nonumber\\
&=a\int\nabla^ru(x)\int \mathbf \Gamma(x-y)\nabla_y^r \Big(u(y)(\frac{1}{\nu}\sigma(y)+\bar{\kappa})^\nu\Big)\mathrm{d}y\mathrm{d}x\nonumber\\
&=a\int\nabla^ru(x)\int \mathbf \Gamma(x-y)\nabla_y^r u(y)(\frac{1}{\nu}\sigma(y)+\bar{\kappa})^\nu\mathrm{d}y\mathrm{d}x\nonumber\\
&\quad+a\int\nabla^ru(x)\int \mathbf \Gamma(x-y)(\nabla_y(u(y)(\frac{1}{\nu}\sigma(y)+\bar{\kappa})^\nu)-\nabla_y^r u(y)(\frac{1}{\nu}\sigma(y)+\bar{\kappa})^\nu \Big)\mathrm{d}y\mathrm{d}x\nonumber\\
&\leq C\delta(\|\nabla^ru\|_{L^2}^2+\|u\|_{H^{s-1}}^2)+a\bar{\kappa}^\nu\|\mathbf \Gamma\|_{L^1}\|\nabla^ru\|_{L^2}^2.
\end{align}

Collecting estimates (\ref{i1}), (\ref{i2}), (\ref{I3}), (\ref{I4z}) and put them into  (\ref{guh}), we obtain that
\begin{align}
\frac{1}{2}\frac{d}{dt}\Big(\|\nabla^r\sigma\|_{L^2}^2+\|\nabla^ru\|_{L^2}^2\Big)&+(\frac{1}{\tau}-2a\bar{\kappa}^\nu\|\mathbf \Gamma\|_{L^1 })\|\nabla^ru\|_{L^2}^2\nonumber\\
&\leq C\delta (\|\nabla^ru\|_{L^2}^2+\|u\|_{H^{s-1}}^2)+C\delta\|\nabla^r \sigma\|_{L^2}^2).
\end{align}
\end{proof}
Now, we will bring forward the dissipation estimate for $\sigma$.
\begin{lem}\label{5.3}For $1\leq r\leq s$,
\begin{equation}\label{hh}
\frac{d}{dt}\int \nabla^{r-1}u\nabla^r\sigma\mathrm{d}x+\frac{\bar{\kappa}}{4}\|\nabla^r \sigma\|_{L^2}^2\leq C\|u\|_{H^{s}}^2+C\delta(\|\nabla^r \sigma\|_{L^2}^2+\|\nabla \sigma\|_{H^{s-1}}^2).
\end{equation}
\end{lem}
\begin{proof} First, we can directly calculate to obtain
\begin{equation}\label{utnt}
\frac{d}{dt}\int \nabla^{r-1}u\nabla^r\sigma\mathrm{d}x=\int \nabla^{r-1}u\nabla^r\sigma_t\mathrm{d}x+\int \nabla^{r-1}u_t \nabla^r\sigma\mathrm{d}x.
\end{equation}

Next, we will estimate the right-hand two terms of the upper equation.
Let $1\leq r\leq s$, applying $\nabla^{r}$ to (\ref{eqn1a}), multiplying it by $\nabla^{r-1}u$ and integrating over $\mathbb{R}^N$ we obtain
\begin{align}\label{s-1quan}
\int \nabla^{r-1}u\nabla^r\sigma_t\mathrm{d}x&=-\bar{
\kappa}\int \nabla\cdot\nabla^ru\cdot\nabla^{r-1}u\mathrm{d}x\nonumber\\
&\quad-\int \nabla^{r}(u\cdot\nabla \sigma)\cdot\nabla^{r-1}u\mathrm{d}x-\frac{1}{\nu}\int \nabla^{r}(\sigma\nabla\cdot u)\cdot\nabla^{s-1}u\mathrm{d}x\\
&\leq C\|\nabla^r u\|_{L^2}^2+C\|\nabla^r u\|_{L^2}\|\nabla^{r-1}(u\cdot\nabla \sigma)\|_{L^2}+C\|\nabla^r u\|_{L^2}\|\nabla^{r-1}(\sigma\nabla\cdot u)\|_{L^2}.\nonumber
\end{align}
In order to get the estimate of (\ref{s-1quan}), we estimate $\|\nabla^{r-1}(u\cdot\nabla \sigma)\|_{L^2}$ and $\|\nabla^{r-1}(\sigma\cdot\nabla u)\|_{L^2}$. By Moser type inequality and Sobolev embedding theorem, we have
\begin{align}\label{ss}
\|\nabla^{r-1}(u\cdot\nabla \sigma)\|_{L^2}&\leq C \|u\|_{L^\infty}\|\nabla^r \sigma\|_{L^2}+C\|\nabla \sigma\|_{L^\infty}\|\nabla^{r-1}u\|_{L^2}\nonumber\\
&\leq C\|u\|_{H^{s-1}}\|\nabla^r \sigma\|_{L^2}+C\|\nabla \sigma\|_{H^{s-1}}\|\nabla^{r-1}u\|_{L^2}\nonumber\\
&\leq C\delta(\|\nabla^r \sigma\|_{L^2}+\|\nabla \sigma\|_{H^{s-1}}).
\end{align}
\begin{align}\label{sss}
\|\nabla^{r-1}(\sigma\cdot\nabla u)\|_{L^2}&\leq \|\sigma\|_{L^\infty}\|\nabla^r u\|_{L^2}+\|\nabla u\|_{L^\infty}\|\nabla^{r-1}\sigma\|_{L^2}\nonumber\\
&\leq C\|\sigma\|_{H^{s-1}}\|\nabla^r u\|_{L^2}+C\|\nabla u\|_{H^{s-1}}\|\nabla^{r-1}\sigma\|_{L^2}\nonumber\\
&\leq C\delta(\|\nabla^ru\|_{L^2}+\|\nabla u\|_{H^{s-1}})
\end{align}
%If $l=0$, then
%\begin{equation}
%\|\sigma\cdot\nabla^r u\|_{L^2}\leq C\|\sigma\|_{L^\infty}\|\nabla^r u\|_{L^2}\leq \delta\|\nabla^r u\|_{L^2}.
%\end{equation}
%If $1\leq l\leq\frac{N}{2}$,
%If $[\frac{N}{2}]\leq l$,
% by H${\rm \ddot{o}}$lder inequality and Lemma \ref{SE}, we have
%\begin{align}
%\|\nabla^{l}u\cdot\nabla^{r-l} \sigma)\|_{L^2}&\leq \|\nabla^{l}u\|_{L^{\frac{2N}{N-2}}}\|\nabla^{r-l} \sigma)\|_{L^{N}}\nonumber\\
%&\leq\|u\|_{L^2}^{\frac{r-l-1}{r}}\|\nabla^r u\|_{L^2}^{\frac{l+1}{r}}\|\nabla^\alpha \sigma\|_{L^2}^{\frac{l+1}{r}}\|\nabla^r \sigma\|_{L^2}^{\frac{r-l-1}{r}}
%\end{align}
%where $\alpha=\frac{Nr}{2(l+1)}$????????
%So, for $1\leq r\leq s$, we have
%\begin{equation}\label{s-1un}
%\|\nabla^{s-1}(u\cdot\nabla n)\|_{L^2}\leq C\delta(\|\nabla^r u\|_{L^2}+\|\nabla^r n\|_{L^2}).

Then, substituting (\ref{ss}) and (\ref{sss}) for (\ref{s-1quan}) and applying Young's inequality, we can deduce that
\begin{align}\label{nt}
\int \nabla^{r-1}u\nabla^r\sigma_t\mathrm{d}x\leq C(\|\nabla^r u\|_{L^2}^2+\|\nabla u\|_{H^{s-1}}^2) +C\delta(\|\nabla^r \sigma\|_{L^2}^2+\|\nabla \sigma\|_{H^{s-1}}^2).
\end{align}

Now, we estimate the second item on the right side of (\ref{utnt}).
Let $1\leq r\leq s$, applying $\nabla^{r-1}$ to (\ref{eqn2a}), multiplying it by $\nabla^r\sigma$ and integrating over $\mathbb{R}^N$ we obtain
\begin{align}
\int \nabla^{r-1}u_t\nabla^r\sigma\mathrm{d}x&+\bar{\kappa}\|\nabla^r \sigma\|_{L^2}^2\nonumber=-\frac{1}{\tau}\int \nabla^{r-1}u\nabla^r\sigma\mathrm{d}x\\
&-\int \nabla^{r-1}(u\cdot\nabla u)\nabla^r\sigma\mathrm{d}x-\frac{1}{\nu}\int \nabla^{r-1}(\sigma\nabla \sigma)\nabla^r\sigma\mathrm{d}x\nonumber\\
&-a\int \nabla^{r-1}
\Big(\int \mathbf \Gamma(x-y)(\frac{1}{\nu}\sigma(y)+\bar{\kappa})^\nu u(x)\mathrm{d}y\Big)\nabla^r\sigma(x)\mathrm{d}x\nonumber\\&+a\int \nabla^{r-1}
\Big(\int \mathbf \Gamma(x-y)(\frac{1}{\nu}\sigma(y)+\bar{\kappa})^\nu u(y)\mathrm{d}y\Big)\nabla^r\sigma(x)\mathrm{d}x
=\sum_{i=1}^{5}I_i.
\end{align}
 Using Young's and Holder's inequality, we have
\begin{equation}\label{a}
I_1=-\frac{1}{\tau}\int \nabla^{r-1}u\nabla^r\sigma\mathrm{d}x\leq C\|\nabla^{r-1}u\|_{L^2}^2+\frac{\bar{\kappa}}{4}\|\nabla^{r}\sigma\|_{L^2}^2.\\
\end{equation}
By a method similar to the estimate for $\|\nabla^{r-1}(u\cdot\nabla \sigma)\|_{L^2}$ and $\|\nabla^{r-1}(\sigma\cdot\nabla u)\|_{L^2}$, we have
\begin{align}
I_2&=-\int \nabla^{r-1}(u\cdot\nabla u)\nabla^r\sigma\mathrm{d}x\leq C(\|\nabla^r u\|_{L^2}^2+\|\nabla u\|_{H^{s-1}}^2) +C\delta\|\nabla^r \sigma\|_{L^2}^2,\\
I_3&=-\frac{1}{\nu}\int \nabla^{r-1}(\sigma\nabla \sigma)\nabla^r\sigma\mathrm{d}x\leq C\delta(\|\nabla^r \sigma\|_{L^2}^2+\|\nabla \sigma\|_{H^{s-1}}^2).
%+\delta\|\nabla^r n\|_{L^2}^2.
\end{align}

Next, we estimate $I_4$. Similar to (\ref{4.14}) and (\ref{4.16}), using Young's inequality and Moser type inequality, we have
\begin{align}
I_4&=-a\int \nabla^r\sigma(x)\nabla^{r-1}_x
\Big(u(x)\int \mathbf \Gamma(x-y)(\frac{1}{\nu}\sigma(y)+\bar{\kappa})^\nu \mathrm{d}y\Big)\mathrm{d}x\nonumber\\
&\leq\|\nabla^r\sigma(x)\|_{L^2}\|\nabla^{r-1}
\Big(u(x) \mathbf \Gamma\ast(\frac{1}{\nu}\sigma+\bar{\kappa})^\nu(x) \|_{L^2}\nonumber\\
&\leq \|\nabla^r\sigma|_{L^2}\|\nabla^{r-1}u\|_{L^2}\|\mathbf \Gamma\ast(\frac{1}{\nu}\sigma+\bar{\kappa})^\nu \|_{L^\infty}+\|\nabla^r\sigma|_{L^2}\|u\|_{L^\infty}\|\mathbf \Gamma\ast(\frac{1}{\nu}\sigma+\bar{\kappa})^\nu \|_{L^2}\nonumber\\
&\leq C \|\mathbf \Gamma\|_{L^1}\|\nabla^r\sigma|_{L^2}\Big(\|\nabla^{r-1}u\|_{L^2}+\|u\|_{H^{s-1}}\|\sigma\|_{H^{r-1}}\Big)\nonumber\\
&\leq C \|u\|_{H^{s-1}}^2+\frac{\bar{\kappa}}{4}\|\nabla^r\sigma\|_{L^2}.
\end{align}
Applying the same method as $I_4$, we can get
\begin{align}
I_5&=a\int \nabla_x^{r-1}
\Big(\int \mathbf \Gamma(x-y)(\frac{1}{\nu}\sigma(y)+\bar{\kappa})^\nu u(y)\mathrm{d}y\Big)\nabla^r\sigma(x)\mathrm{d}x\nonumber\\
&=a\int
\int \mathbf \Gamma(x-y) \nabla_y^{r-1}\Big((\frac{1}{\nu}\sigma(y)+\bar{\kappa})^\nu u(y)\Big)\mathrm{d}y\nabla^r\sigma(x)\mathrm{d}x\nonumber\\
&\leq C \|u\|_{H^{s-1}}^2+\frac{\bar{\kappa}}{4}\|\nabla^r\sigma|_{L^2}.\label{i4}
\end{align}

Collecting all estimates of $I_i$, we have
\begin{align}\label{ut}
\int \nabla^{r-1}u_t\nabla^r\sigma\mathrm{d}x+\frac{\bar{\kappa}}{4}\|\nabla^r \sigma\|_{L^2}^2\leq C\|u\|_{H^{s-1}}^2+C\delta(\|\nabla^r \sigma\|_{L^2}^2+\|\nabla \sigma\|_{H^{s-1}}^2).
\end{align}
Combining (\ref{nt}) and ({\ref{ut}}) to obtain
\begin{equation}
\frac{d}{dt}\int \nabla^{r-1}u\nabla^r\sigma\mathrm{d}x+\frac{\bar{\kappa}}{4}\|\nabla^r \sigma\|_{L^2}^2\leq C\|u\|_{H^{s}}^2+C\delta(\|\nabla^r \sigma\|_{L^2}^2+\|\nabla \sigma\|_{H^{s-1}}^2).
\end{equation}
%Combining (\ref{nt}), (\ref{ut}), and selecting the appropriate $\varepsilon_1$, we obtain
%\begin{align}
%\frac{d}{dt}\int \nabla^{s-1}u\nabla^rn\mathrm{d}x+\frac{\bar{\kappa}^2}{2}\|\nabla^r \sigma\|_{L^2}^2\leq C\| u\|_{H^r}^2+.
%\end{align}
\end{proof}
\subsection{The proof of global existence}

In this subsection, we construct the global-in-time solution by combining the local existence theory.% in theorem \ref{local} and uniform a priori estimate in subsection 5.1.

We sum up the estimate (\ref{guuh}) in Lemma \ref{5.2} form $r=1$ to $s$, and then add the estimate (\ref{gul2}) in Lemma \ref{5.1}, since $\delta$ is small and conditions (\ref{mu}), we can deduce that there exists $\varepsilon_1>0, C_1>0$ such that
\begin{equation}\label{guh3}
\frac{d}{dt}\Big(\|\sigma\|_{H^s}^2+\|u\|_{H^s}^2\Big)+\varepsilon_1\|u\|_{H^s}^2\leq C_1\delta \sum_{r=1}^s\|\nabla^{r} \sigma\|_{L^2}^2.
\end{equation}

Summing up the estimates (\ref{hh}) in Lemma \ref{5.3} from $r=1$ to $s$, for sufficient small $\delta$, there exist $\varepsilon_2>0, C_2>0$, such that
\begin{equation}\label{hhs}
\frac{d}{dt}\sum_{r=1}^s\int \nabla^{r-1}u\nabla^r\sigma\mathrm{d}x+\varepsilon_2\|\nabla \sigma\|_{H^{s-1}}^2\leq C_2\|u\|_{H^s}^2.
\end{equation}
Multiplying (\ref{hhs}) by $\dfrac{2C_1\delta}{\varepsilon_2}$, adding it to (\ref{guh3}), since $\delta$ is small, then there exits a constant $\varepsilon_3>0$ such that

\begin{equation}\label{guh31}
\frac{d}{dt}\Big(\|\sigma\|_{H^s}^2+\|u\|_{H^s}^2+\dfrac{2C_1\delta}{\varepsilon_2}\sum_{r=1}^s\int \nabla^{r-1}u\nabla^r\sigma\mathrm{d}x\Big)
+\varepsilon_3\Big(\|u\|_{H^s}^2+\|\nabla \sigma\|_{H^{s-1}}^2\Big)\leq 0.
\end{equation}

Note that there exist constant $C_3>0$ such that

\begin{equation}\label{unh3}
C_3^{-1}\Big(\|\sigma\|_{H^s}^2+\|u\|_{H^s}^2\Big)\leq\Big(\|\sigma\|_{H^s}^2+\|u\|_{H^s}^2+\dfrac{2C_1\delta}{\varepsilon_2}\sum_{r=1}^s
\int \nabla^{r-1}u\nabla^r\sigma\mathrm{d}x\Big)\leq C_3\Big(\|\sigma\|_{H^s}^2+\|u\|_{H^s}^2\Big).
\end{equation}

Integrating directly in time, with help  of (\ref{unh3}), we obtain

\begin{equation}\label{guh31}
\underset{0\leq \tilde{t} \leq t}{sup}\Big(\|\sigma\|_{H^s}^2(\tilde{t})+\|u\|_{H^s}^2(\tilde{t})\Big)
+\int_{0}^{t}\Big(\|u\|_{H^s}^2(\tilde{t})+\|\nabla \sigma\|_{H^{s-1}}^2(\tilde{t})d\tilde{t}\Big)\leq C_3^2\Big(\|\sigma_0\|_{H^s}^2+\|u_0\|_{H^s}^2\Big).
\end{equation}

Finally, we can  use Theorem \ref{local} and (\ref{guh31}) to prove the global existence of classical solutions for (\ref{eqn1a})-(\ref{eqn2a}). Applying the local existence of the classical solution in Theorem \ref{local}, we deduce that for any positive constant $\delta_0<\delta$, there is a positive constant $T_0$ depending only on $\delta_0,\ \delta$ such that if $\|\sigma_0\|_{H^s}+\|u_0\|_{H^s}<\delta_0$, then the solution of the Cauchy problem
(\ref{eqn1a})-(\ref{eqn3a}) satisfies
\begin{equation}
\underset{0\leq t\leq T_0}{\sup}\|\sigma\|_{H^s}+\|u\|_{H^s}<\delta.
\end{equation}
Then,  we choose that
\begin{equation}\label{oo}
%\|\sigma_0\|^2_{H^s}+\|u_0\|^2_{H^s}<
\delta_0=\frac{\delta}{\sqrt{2(1+C_4^2)}},
\end{equation}
where $\delta$ and $C_4^2$ are given in (\ref{xiao}) and (\ref{guh31}), respectively. Let us define the maximal existence time $T_{max}>0$ of the system (\ref{eqn1a})-(\ref{eqn2a}) by
\begin{equation}
T_{max}:=\sup\{t\geq 0: \underset{0\leq t\leq T_0}{\sup}\|\sigma\|_{H^s}+\|u\|_{H^s}<\delta\}.
\end{equation}
Suppose $T_{max}<\infty$, then we can use the continuation argument and (\ref{guh31}) to get
\begin{equation}
\delta^2=\underset{0\leq t\leq T_{max}}{\sup}\|\sigma\|^2_{H^s}+\|u\|^2_{H^s}\leq C_4^2\Big(\|\sigma_0\|_{H^s}^2+\|u_0\|_{H^s}^2\Big)<C_4^2\frac{\delta^2}{2(1+C_4^2)}<\frac{\delta^2}{2}.
\end{equation}
This is a contradiction, hence, we can conclude that $T_{max}=\infty$.

%Further, applying Gronwall's inequality, we can get from $(\ref{guh31})$
%\begin{equation}\label{zh}
%\|u(\cdot,t)\|_{H^s}+\|\nabla\sigma(\cdot, t)\|_{H^{s-1}}\leq Ce^{-Ct}.
%\end{equation}

%We can get the  boundedness of $\|\sigma\|_{L^2}$ in (\ref{guh31}), and use Sobolev embedding to obtain the decay of $\|\nabla \sigma\|_{q}$ in (\ref{zh}), where $2\leq q\leq\infty$. Then applying Gagliardo-Nirenberg inequality, we have the decay of $\|\sigma(\cdot,t)\|_{L^p}$ for $p>2 $ ,
%\begin{equation}
%\|\sigma(\cdot,t)\|_{L^p}\leq \|\sigma(\cdot,t)\|_{L^2}^{\theta}\|\nabla \sigma(\cdot,t)\|_{L^q}^{1-\theta}\leq Ce^{-C(1-\theta)t}\leq Ce^{-Ct}.
%\end{equation}
%where $0<\theta<1$ and $q>p$.

In summary, we have completed the proof of Theorem 2.2.

\begin{appendix}
\section{inequality}

In the appendix, we present several lemmas used in the existence proof in Sections 3 and 4.
\begin{lem}\label{moser}{\rm(Moser type inequality)}For any pair of functions $f,\ g\in(H^k)\cap L^\infty(\mathbb{R}^N)$, we have
\begin{eqnarray}
\|\nabla^{s}(fg)\|_{L^2}\leq C(\|\nabla^r f\|_{L^{2}}\|g\|_{L^\infty}+\|f\|_{L^{\infty}}\|\nabla^k g\|_{L^2}).
\end{eqnarray}
Furthermore if $\nabla f\in L^{\infty}(\mathbb{R}^N)$ we have
\begin{eqnarray}
\|\nabla^{s}(fg)-f\nabla^{s}g\|_{L^2}\leq C(\|\nabla f\|_{L^{\infty}}\|\nabla^{s-1}g\|_{L^2}+\|g\|_{L^{\infty}}\|\nabla^r f\|_{L^2}).
\end{eqnarray}
\end{lem}
\begin{proof}
See Lemma 3.4 in \cite{MAJ}.
\end{proof}

\begin{lem}\label{Young}{\rm(Young's inequality)} Let $p,q,r\geq1$ and $1/p+1/q+1/r=2$. Let $f\in L^{p},g\in L^{q}$ and $h\in L^{r}$. Then
\begin{align}
|\int f(x)(g\ast h)(x)dx|=|\int\int f(x)g(x-y)h(x)|\mathrm{d}x\mathrm{d}y\nonumber\\
\leq C_{p,q,r,N}\|f\|_{L^p}\|g\|_{L^q}\|h\|_{L^r}.
\end{align}
\end{lem}

\begin{proof}
See Theorem 4.2 in \cite{Elliott}.
\end{proof}

\begin{lem}\label{GN}{\rm(Gagliardo-Nirenberg inequality)}
Let $f\in W_{0}^{1,q}\cap L^r$ for some $r\leq1$. There exists a constant $C$ depending upon $N, P, r$ such that
\begin{equation}
\|u\|_{L^p}\leq C\|\nabla u\|_{L^q}^\theta\|u\|_{L^r}^{1-\theta},
\end{equation}
where $\theta\in[0,1]$ and $p,q\leq1$ are linked by
\begin{equation}
\theta=\Big(\frac{1}{r}-\frac{1}{p}\Big)\Big(\frac{1}{N}-\frac{1}{q}+\frac{1}{r}\Big)^{-1}.
\end{equation}
\begin{proof}
See Theorem 1.1 of the Chapter 10 in \cite{ED}.
\end{proof}
\end{lem}

\end{appendix}

%%%%%%%%%%%%%%%%%%%%%%%%%%%%%%%%%%%%

\end{document}